\documentclass[11pt,a4paper]{amsart}
\usepackage[T1]{fontenc}
\usepackage[utf8]{inputenc}
\usepackage{lmodern}
\usepackage{amsmath,amssymb,amsthm,mathtools}
\usepackage{microtype}
\usepackage[margin=27mm]{geometry}
\usepackage{enumitem}
\usepackage{xurl}
\usepackage[hidelinks,unicode,hypertexnames=false]{hyperref}
\hypersetup{pdftitle={Second-Order Spectral Loss under Conditional Expectations},
  pdfauthor={Trung Hoa Dinh, Cong Trinh Le, Trung Dung Vuong}}
\allowdisplaybreaks[1]
\DeclareMathOperator{\Tr}{Tr}
\DeclareMathOperator{\spec}{spec}

\theoremstyle{plain}
\newtheorem{theorem}{Theorem}[section]
\newtheorem{lemma}[theorem]{Lemma}
\newtheorem{proposition}[theorem]{Proposition}
\newtheorem{corollary}[theorem]{Corollary}
\theoremstyle{definition}
\newtheorem{definition}[theorem]{Definition}
\theoremstyle{remark}
\newtheorem{remark}[theorem]{Remark}
\newtheorem*{remark*}{Remark}

\title[Positive planar measures and corrected BMV]
{Positive Planar Measures for Spectral Pinching
and the Corrected BMV Conjecture}
\author{Trung Hoa Dinh}
\address{Department of Mathematics and Statistics, Troy University, Troy,
Alabama 36082, USA}
\email{thdinh@troy.edu}
\author{Cong Trinh Le}
\address{Department of Mathematics and Statistics, Quy Nhon University,
170 An Duong Vuong, Quy Nhon, Vietnam}
\email{lecongtrinh@qnu.edu.vn}
\author{Trung Dung Vuong}
\address{High School for the Gifted, VNUHCM, 153 Nguyen Chi Thanh, An Dong
Ward, Ho Chi Minh City, Vietnam; Vietnam National University Ho Chi Minh
City, Linh Xuan Ward, Ho Chi Minh City, Vietnam}
\email{vtdung@ptnk.edu.vn}
\subjclass[2020]{Primary 46L51, 47A55, 47A63; Secondary 47B10, 15A45, 60G42, 81P45}
\keywords{conditional expectation, second-order spectral loss,
spectral pinching, corrected BMV conjecture, positive planar measure,
tracial joint spectral measure, Lie--Trotter formula}
\date{}

\begin{document}
\begin{abstract}
We prove the corrected Bessis--Moussa--Villani pinching conjecture
for all symmetrized word averages. For Hermitian matrices $A,B$ and
$D=E_A(B)$, we construct a positive compactly supported measure
$\zeta_{A,B}$ on $\mathbb R^2$ such that
\[
\Tr f(xA+yB)-\Tr f(xA+yD)
=y^2\int_{\mathbb R^2}f''(xa+yb)\,d\zeta_{A,B}(a,b).
\]
Its total mass is $\frac12\|B-D\|_2^2$. This representation proves
the corrected conjecture, including its equality cases, and gives
a positive bivariate moment representation for the coefficient
differences. We also derive complete monotonicity of normalized
exponential pinching defects and a comparison with Lie--Trotter
approximations. A one-dimensional second-order loss measure is
obtained for trace-preserving conditional expectations on finite
tracial von Neumann algebras. Finally, we relate $\zeta_{A,B}$
to Hein\"avaara's tracial joint spectral measure by an explicit
positive radial transform.
\end{abstract}
\maketitle

\section{Introduction}

The Bessis--Moussa--Villani (BMV) conjecture arose in quantum
statistical mechanics \cite{BMV}. In its coefficient formulation
\cite{LiebSeiringer}, it asks for the nonnegativity of symmetrized
word traces in two positive semidefinite matrices. Stahl proved the
conjecture \cite{Stahl}. A later refinement proposed, in particular,
that the normalized trace average $\mathcal A_{n,m}(A,B)$ over all
words with $n$ letters $A$ and $m$ letters $B$ should be bounded above
by $\Tr(A^nB^m)$. Counterexamples in \cite{Cha} show that this upper
bound is false. For us, the failure raised a further question: what
makes these counterexamples possible, and is there a natural
comparison that remains valid?

The first author studied this question in \cite{Dinh}. The key issue
is that the clustered word $\Tr(A^nB^m)$ need not capture the common
commuting part of the matrices. Contributions from off-diagonal
spectral blocks can connect different eigenspaces, and their
arrangement in mixed words can defeat the clustered bound. This
suggested keeping the part of $B$ that commutes with $A$ and measuring
the contribution of what is discarded. If
$A=\sum_{\alpha\in\spec A}\alpha P_\alpha$, the natural map is the
spectral pinching
\[
 E_A(B)=\sum_{\alpha\in\spec A}P_\alpha BP_\alpha.
\]
It is also the trace-preserving conditional expectation onto the
commutant of $A$, a familiar operation in operator theory and quantum
information. Using this map, \cite{Dinh} proposed the
\emph{corrected BMV conjecture}
\begin{equation}\label{eq:intro-bmv}
 \mathcal A_{n,m}(A,B)\ge
 \mathcal A_{n,m}\bigl(A,E_A(B)\bigr),
 \qquad A,B\ge0,\quad n,m\ge0.
\end{equation}
The right-hand side equals $\Tr(A^nE_A(B)^m)$ and is nonnegative.
Thus, the conjecture strengthens the original BMV positivity while
avoiding the failed clustered upper bound. The case $m=2$ was proved
in \cite{Dinh}, but the general case remained open.

Our initial aim was to settle \eqref{eq:intro-bmv} for all $n,m$.
Instead of treating the word averages one at a time, we looked for an
object representing their entire difference. This led us back to a
more basic question: how much spectral information is lost when a
matrix is replaced by its conditional expectation? The correction
in \eqref{eq:intro-bmv} had identified a natural comparison, but a
proof for all coefficients seemed to require a representation of the
loss itself.

An important guide was the work of Otte Hein\"avaara
\cite{HeinTracial}. He constructed a positive tracial joint spectral
measure for every Hermitian pair $(A,B)$. Through an integral
transform, this measure represents trace functionals of the pencil
$xA+yB$ and leads, among other results, to the classical BMV theorem
and to $L^p$ models for two-dimensional Schatten geometry. We studied
this approach because it replaces many separate trace inequalities
by a single positive measure. However, \eqref{eq:intro-bmv} concerns
the \emph{difference} between the pencils $xA+yB$ and
$xA+yE_A(B)$. The existence of a positive measure for each pencil
does not imply that their difference is positive. The representation
in \cite{HeinTracial} therefore does not, by itself, establish the
corrected conjecture. We needed a measure adapted to conditional
expectation, rather than to a pair of matrices alone.

This became the main problem of the paper. Let $E$ be a
trace-preserving conditional expectation on a finite tracial von
Neumann algebra $(\mathcal M,\tau)$. For a bounded selfadjoint $X$,
trace Jensen gives $\tau f(X)\ge\tau f(EX)$ for every scalar convex
$f$ \cite{BrownKosaki,HaradaKosaki}. The first variation of the trace
functional at $EX$ vanishes in the direction $X-EX$, so the loss is
naturally second order. We show that there is a canonical positive
compactly supported measure $\nu_{E,X}$ such that
\begin{equation}\label{eq:intro-one}
 \tau f(X)-\tau f(EX)
 =\int_{\mathbb R}f''(s)\,d\nu_{E,X}(s),
 \qquad
 \nu_{E,X}(\mathbb R)=\frac12\|X-EX\|_2^2.
\end{equation}
This measure works for every suitable $C^2$ function $f$. The
one-dimensional formula is related to second-order spectral-shift
theory \cite{Koplienko,Peller,GPS,DykemaSkripka}. But positivity for
each individual direction does not answer the question that emerged
from the BMV problem: can all these losses be represented by one
\emph{positive measure on a plane}?

For spectral pinching, we prove that the answer is yes. Given
Hermitian $A,B$, put $D=E_A(B)$. There is a canonical positive
compactly supported measure $\zeta_{A,B}$ on $\mathbb R^2$ for which
\begin{equation}\label{eq:intro-planar}
 \Tr f(xA+yB)-\Tr f(xA+yD)
 =y^2\int_{\mathbb R^2}f''(xa+yb)\,d\zeta_{A,B}(a,b)
\end{equation}
for all real $x,y$ and suitable $C^2$ functions $f$. Its total mass
is $\frac12\|B-D\|_2^2$, and it vanishes exactly when $A$ and $B$
commute. The construction uses a selected-root Herglotz measure for a
quotient matrix pencil. Removing its atomic part, which is precisely
the part retained by pinching, leaves a positive measure representing
the discarded spectral information. The same planar measure gives all
the second-order losses in \eqref{eq:intro-planar}.

The corrected BMV conjecture is now a consequence of this formula.
For $A,B\ge0$, $n\ge0$, and $m\ge2$, we obtain
\begin{equation}\label{eq:intro-bmv-gap}
 \mathcal A_{n,m}(A,B)-\mathcal A_{n,m}(A,D)
 =m(m-1)\int_{\mathbb R^2}a^n b^{m-2}
 \,d\zeta_{A,B}(a,b)\ge0.
\end{equation}
The cases $m=0,1$ are equalities. Thus, the conjecture that
motivated the work follows from a result about spectral loss that
applies to much more than word averages. In particular, the formula
also gives positive mixed moments and complete monotonicity of the
normalized exponential pinching defect. This change of viewpoint, from proving
one family of inequalities to understanding the loss under pinching,
is the reason for the general formulation of the paper.

The relation with Hein\"avaara's theory is also more precise at this
level. His tracial joint spectral measure describes trace geometry
of a Hermitian pair. Our measure describes the \emph{relative
second-order loss} when one member of the pair is replaced by its
pinching. They are different objects, not two descriptions of the
same measure. Nevertheless, for positive semidefinite matrices, the
finite comparison measure obtained from Hein\"avaara's integral
representation is a positive radial transform of $\zeta_{A,B}$
(Theorem~\ref{thm:5.9}). This identifies a connection between the
two theories while making clear what additional positivity has to
be proved for the pinching comparison.

Section~\ref{sec:oned} establishes the one-dimensional loss formula.
Section~\ref{sec:pinching} constructs the positive planar measure for
spectral pinching. Section~\ref{sec:applications} proves the corrected
BMV conjecture and gives moment and Lie--Trotter consequences.
Section~\ref{sec:joint-spectral} compares the planar measure with
Hein\"avaara's tracial joint spectral measure.

\subsection*{Declaration of AI-assisted research and writing}

The use of AI in this research developed in stages.
The first author began by studying the failure of the
refined BMV conjecture. He used OpenAI's ChatGPT to
examine known counterexamples and explore possible
explanations for their failure. This led him to consider
the commuting part retained by spectral pinching,
formulate the corrected BMV conjecture in \cite{Dinh},
and prove its first case $m=2$. The mathematical
question and the choice of spectral pinching were
decisions of the first author.

For the present paper, the authors used ChatGPT as
a research tool to speed up the study of related
literature, mathematical calculations, and the
exploration and checking of possible arguments.
This included studying Hein\"avaara's approach
in \cite{HeinTracial} and investigating whether
the conditional-expectation idea from \cite{Dinh}
could lead to a positive second-order measure
representation.

The research direction, mathematical framework,
and organization of the paper were determined by
the authors. ChatGPT also assisted with manuscript
preparation and language editing. The authors
independently verified the mathematical statements,
proofs, and references, and take full responsibility
for the content of the paper.

\section{One-dimensional second-order loss under conditional expectations}
\label{sec:oned}

Let $(\mathcal M,\tau)$ be a finite tracial von Neumann algebra,
where $\tau$ is faithful and normal, with
$0<\tau(1)<\infty$.   
Let $E:\mathcal M\longrightarrow\mathcal N$
be the trace-preserving conditional expectation onto a
von Neumann subalgebra $\mathcal N$.
For a self-adjoint $X\in\mathcal M$, put $D=EX$.
We write $\|Y\|_2^2=\tau(Y^*Y)$.

\begin{definition}\label{def:4.1}
For $s\in\mathbb R$, set
\[
h_{E,X}(s)
=\tau\bigl((X-s1)_+\bigr)
-\tau\bigl((D-s1)_+\bigr).
\]
Define the measure $\nu_{E,X}$ by
\[
d\nu_{E,X}(s)=h_{E,X}(s)\,ds.
\]
\end{definition}

The function $h_{E,X}$ is the difference of the
stop-loss transforms of the spectral distributions
of $X$ and $D$. The measure $\nu_{E,X}$ is also
the Koplienko spectral shift measure associated
with the perturbation $X-D$ of $D$.

\begin{theorem}\label{thm:oned-loss}
The measure $\nu_{E,X}$ is finite, positive, and
compactly supported. For every $C^2$ function $f$
on an interval containing $\spec X$, we have
\begin{equation}\label{eq:18}
\tau f(X)-\tau f(D)
=\int_{\mathbb R}f''(s)\,d\nu_{E,X}(s).
\end{equation}
Moreover,
\begin{equation}\label{eq:19}
\nu_{E,X}(\mathbb R)
=\frac12\bigl(\tau(X^2)-\tau(D^2)\bigr)
=\frac12\|X-D\|_2^2.
\end{equation}
Consequently,
\[
\nu_{E,X}=0
\quad\Longleftrightarrow\quad X=EX.
\]
\end{theorem}

\begin{proof}
For every $s\in\mathbb R$, the variational formula
for the positive part gives
\begin{align*}
\tau\bigl((D-s1)_+\bigr)
&=\sup_{\substack{0\le H\le1\\H\in\mathcal N}}
  \tau\bigl((D-s1)H\bigr)\\
&=\sup_{\substack{0\le H\le1\\H\in\mathcal N}}
  \tau\bigl((X-s1)H\bigr)\\
&\le\tau\bigl((X-s1)_+\bigr).
\end{align*}
Indeed, for every $H\in\mathcal N$, the bimodule
property of $E$ gives
\[
E((X-s1)H)=(EX-s1)H=(D-s1)H.
\]
Since $\tau\circ E=\tau$, it follows that
\[
\tau((X-s1)H)=\tau((D-s1)H).
\]
The last inequality follows from the variational
formula for the positive part in $\mathcal M$.
Thus, $h_{E,X}(s)\ge0$ for every $s\in\mathbb R$.

Since $E$ is unital and positive,
\[
\lambda_{\min} (X)1\le D\le \lambda_{\max} (X)1.
\]
If $s\ge\max\spec X$, then
\[
(X-s1)_+=(D-s1)_+=0.
\]
If $s\le\min\spec X$, then
\[
\begin{aligned}
&\tau\bigl((X-s1)_+\bigr)
-\tau\bigl((D-s1)_+\bigr) =\tau(X)-\tau(D)=0.
\end{aligned}
\]
Furthermore, for every $r,s,t\in\mathbb R$,
\[
|(t-s)_+-(t-r)_+|\le|s-r|.
\]
By functional calculus,
\[
|h_{E,X}(s)-h_{E,X}(r)|
\le2\tau(1)|s-r|.
\]
Hence $h_{E,X}$ is continuous and vanishes
outside the spectral interval of $X$.
Therefore, $\nu_{E,X}$ is a finite positive
measure with compact support.

Let $L=\lambda_{\min} (X)$ and $U=\lambda_{\max} (X)$.
For every $t\in[L,U]$, Taylor's formula gives
\[
f(t)=f(L)+f'(L)(t-L)
+\int_L^U(t-s)_+f''(s)\,ds.
\]
Applying functional calculus to $X$ and $D$,
taking traces, and using $\tau(X)=\tau(D)$,
we obtain
\[
\begin{aligned}
\tau f(X)-\tau f(D)
&=\int_L^U f''(s)
\bigl[
\tau\bigl((X-s1)_+\bigr) 
-\tau\bigl((D-s1)_+\bigr)
\bigr]\,ds\\
&=\int_{\mathbb R}f''(s)\,d\nu_{E,X}(s).
\end{aligned}
\]
Here the interchange of trace and integration
is justified by the boundedness of $f''$
on $[L,U]$ and the finiteness of $\tau$.
This proves \eqref{eq:18}.

Taking $f(t)=t^2/2$ in \eqref{eq:18} gives
\[
\nu_{E,X}(\mathbb R)
=\frac12\bigl(\tau(X^2)-\tau(D^2)\bigr).
\]
Since $E$ is the orthogonal projection of
$L^2(\mathcal M,\tau)$ onto
$L^2(\mathcal N,\tau)$, we have
\[
\tau(XD)=\tau(D^2).
\]
Consequently,
\[
\|X-D\|_2^2
=\tau(X^2)-\tau(D^2),
\]
which proves \eqref{eq:19}.
The last assertion follows from the
faithfulness of $\tau$.
\end{proof}

\section{Spectral pinching and positive planar compatibility}
\label{sec:pinching}

Section~\ref{sec:oned} provides a positive measure
representing the second-order trace loss under a
conditional expectation. We now consider spectral
pinching and prove that the trace losses for an entire
two-parameter family of Hermitian matrix pencils
can be represented by a single positive measure
on $\mathbb R^2$.


Let $\Tr$ denote the usual matrix trace, and write
$\|X\|_2^2=\Tr(X^*X)$. For a matrix $M\in M_d(\mathbb C)$,
let $\lambda_1(M),\ldots,\lambda_d(M)$ denote its
eigenvalues, counted with algebraic multiplicity. For a Hermitian matrix
$A=\sum_{\alpha\in\spec A}\alpha P_\alpha,$
the spectral pinching map is defined by
\[
E_A(X)=\sum_{\alpha\in\spec A}P_\alpha XP_\alpha.
\]
This is the trace-preserving conditional expectation
onto the commutant $\{A\}'$.
For Hermitian matrices $A,B$, put $D=E_A(B)$.

For $a\notin\spec A$ and $b\in\mathbb R$, define
\begin{equation}\label{eq:1}
k_{A,B}(a,b)
=\frac{1}{2\pi}\sum_{j=1}^d
\left|
\Im\lambda_j\bigl((B-bI)(A-aI)^{-1}\bigr)
\right|.
\end{equation}
Set
\[
d\zeta_{A,B}(a,b)=k_{A,B}(a,b)\,da\,db.
\]
Notice that the values of $k_{A,B}$ on the lines
$a\in\spec A$ do not affect this measure.

The main theorem of the paper is as follows.
\begin{theorem}\label{thm:2.1}
The measure $\zeta_{A,B}$ is finite, positive,
and supported in
\[
[\lambda_{\min}(A),\lambda_{\max}(A)]
\times
[\lambda_{\min}(B),\lambda_{\max}(B)].
\]
For every $x,y\in\mathbb R$ and every $C^2$
function $f$ defined on an interval containing
the spectrum of $xA+yB$, we have
\begin{equation}\label{eq:2}
\Tr f(xA+yB)-\Tr f(xA+yD)
=y^2\int_{\mathbb R^2}
f''(xa+yb)\,d\zeta_{A,B}(a,b).
\end{equation}
Moreover,
\begin{equation}\label{eq:3}
\zeta_{A,B}(\mathbb R^2)
=\frac12\|B-D\|_2^2.
\end{equation}
Consequently, $\zeta_{A,B}=0$ if and
only if $AB=BA$.
\end{theorem}

To prove Theorem~\ref{thm:2.1}, we need some
preliminary results. Our approach is related to
the work of Hein\"avaara \cite{HeinTracial},
who constructed a positive tracial joint spectral
measure for a pair of Hermitian matrices.
The main difference is that we consider the trace
difference between $xA+yB$ and $xA+yE_A(B)$,
rather than the trace functional of a single
matrix pencil. The positivity of the tracial
joint spectral measure does not directly imply
the positivity of this difference.

Our proof uses the Herglotz measure associated
with a quotient matrix pencil. We first identify
its atomic masses in terms of the spectral
compressions onto the eigenspaces of $A$.
Removing these atoms gives a positive measure
whose marginal represents the trace loss under
spectral pinching. We then use the shear identity
for matrix pencils to obtain the representation
for all real $x,y$.


For $\Im z>0$, let $m_{S,T}(z)$ denote the sum
of the eigenvalues of $S(zI-T)^{-1}$ 
lying in the open lower half-plane, counted
with algebraic multiplicity. If there are
no such eigenvalues, we set $m_{S,T}(z)=0$. We write
$
S_+=\frac12(S+|S|)$, where $|S|=(S^*S)^{1/2},$
for the positive part of a   matrix $S$.

\begin{lemma}\label{lem:2.4}
For Hermitian $S,T\in M_d(\mathbb C)$, there
exists a finite positive measure $\nu_{S,T}$
on $\mathbb R$ such that
\begin{equation}\label{eq:6}
m_{S,T}(z)
=\int_{\mathbb R}\frac{d\nu_{S,T}(v)}{z-v},
\qquad \Im z>0.
\end{equation}
Moreover,
\begin{equation}\label{eq:7}
\operatorname{supp}\nu_{S,T}
\subseteq [\lambda_{\min}(T),\lambda_{\max}(T)],
\qquad
\nu_{S,T}(\mathbb R)=\Tr S_+,
\end{equation}
and
\begin{equation}\label{eq:8}
\int_{\mathbb R}v\,d\nu_{S,T}(v)
=\Tr(S_+T).
\end{equation}
For each $\lambda\in\spec T$, let $P_\lambda$
be the corresponding spectral projection.
Then
\begin{equation}\label{eq:9}
\nu_{S,T}(\{\lambda\})
=\Tr(P_\lambda SP_\lambda)_+.
\end{equation}
These are the only atoms of $\nu_{S,T}$.
The measure has no singular continuous part,
and its absolutely continuous part satisfies
\begin{equation}\label{eq:10}
\frac{d\nu_{S,T}^{\mathrm{ac}}}{dv}
=\frac1{2\pi}\sum_{j=1}^d
\left|
\Im\lambda_j\bigl(S(vI-T)^{-1}\bigr)
\right|
\end{equation}
for almost every $v\notin\spec T$.
\end{lemma}

\begin{proof}
If $S=0$, the result is immediate. Suppose
$S\ne0$, and write
\[
R(z)=(zI-T)^{-1},\qquad \Im z>0.
\]
We first show that the sum defining $m_{S,T}$
is holomorphic.

With respect to the orthogonal decomposition
$\operatorname{ran}S\oplus\ker S$, write
\[
S=\begin{pmatrix}S_0&0\\0&0\end{pmatrix},
\qquad
R(z)=\begin{pmatrix}
R_{11}(z)&R_{12}(z)\\
R_{21}(z)&R_{22}(z)
\end{pmatrix},
\]
where $S_0$ is invertible. Then
\[
SR(z)=
\begin{pmatrix}
S_0R_{11}(z)&S_0R_{12}(z)\\
0&0
\end{pmatrix}.
\]
For every nonzero $u\in\operatorname{ran}S$,
the resolvent identity gives
\[
\Im\langle u,R_{11}(z)u\rangle
=-(\Im z)\|R(z)u\|^2<0.
\]
In particular, $R_{11}(z)$ is invertible.

Suppose that $S_0R_{11}(z)$ has a real
eigenvalue $\kappa$ with eigenvector $u$.
Then
\[
\langle u,R_{11}(z)u\rangle
=\kappa\langle u,S_0^{-1}u\rangle
\in\mathbb R,
\]
a contradiction. Thus, the nonzero
eigenvalues of $SR(z)$ never lie on the
real axis. Its remaining eigenvalues
are identically zero.

As $z$ tends to infinity in the upper
half-plane, the nonzero eigenvalues of
$SR(z)$ approach those of $S/z$.
Consequently, exactly
$r=\operatorname{rank}S_+$ eigenvalues
lie in the lower half-plane, counted
with algebraic multiplicity.
Since these eigenvalues remain separated
from the others, their sum is holomorphic
by the Riesz projection formula.
In particular,
\[
\Im m_{S,T}(z)\le0,\qquad \Im z>0.
\]

We next determine the first two terms
of its expansion at infinity. Put
\[
M(w)=S(I-wT)^{-1}
=S+wST+O(w^2),
\]
and let $P(w)$ be the Riesz projection
associated with the eigenvalues of $M(w)$
that converge to the positive eigenvalues
of $S$. Writing $P_+=\mathbf1_{(0,\infty)}(S),$ 
we have $P(0)=P_+$ and
\[
m_{S,T}(z)
=z^{-1}\Tr\bigl(M(z^{-1})P(z^{-1})\bigr).
\]
Differentiating $P(w)^2=P(w)$ at $w=0$
shows that $P'(0)$ is off-diagonal
with respect to
$P_+\oplus(I-P_+)$. Since $S$ is
block diagonal for this decomposition, $\Tr(SP'(0))=0.$ 
It follows that
\[
\begin{aligned}
\Tr(M(w)P(w))
&=\Tr(SP_+)
+w\Tr(STP_+)+O(w^2)\\
&=\Tr S_+
+w\Tr(S_+T)+O(w^2).
\end{aligned}
\]
Therefore,
\begin{equation}\label{eq:11}
m_{S,T}(z)
=\frac{\Tr S_+}{z}
+\frac{\Tr(S_+T)}{z^2}
+O(z^{-3}).
\end{equation}
This argument also covers the case
$r=0$, when $m_{S,T}=0$.

Since $-m_{S,T}$ is a Herglotz function,
the Herglotz representation theorem
\cite{GesztesyTsekanovskii}, together
with \eqref{eq:11}, gives a finite
positive measure $\nu_{S,T}$ satisfying
\eqref{eq:6}. Its total mass is
$\Tr S_+$.

For real $z$ outside the spectral
interval of $T$, the matrix
$S(zI-T)^{-1}$ has real spectrum,
since $zI-T$ is definite.
The sum $m_{S,T}$ extends analytically
across these intervals and is real
there. Hence Stieltjes inversion gives
\[
\operatorname{supp}\nu_{S,T}
\subseteq[\lambda_{\min}(T),\lambda_{\max}(T)].
\]
Comparison of the coefficient of
$z^{-2}$ in \eqref{eq:11} with the
expansion of \eqref{eq:6} gives
\eqref{eq:8}. This proves
\eqref{eq:7} and \eqref{eq:8}.

We now determine the atoms of
$\nu_{S,T}$. Fix $\lambda\in\spec T$
and let $z=\lambda+i\varepsilon$.
Since
\[
\varepsilon(zI-T)^{-1}
\longrightarrow-iP_\lambda,
\qquad \varepsilon\downarrow0,
\]
we obtain
\[
\varepsilon S(zI-T)^{-1}
\longrightarrow-iSP_\lambda.
\]
The nonzero eigenvalues of
$SP_\lambda$ coincide with those
of the Hermitian matrix
$P_\lambda SP_\lambda$.
By continuity of eigenvalues,
counted with algebraic multiplicity,
the positive eigenvalues of
$P_\lambda SP_\lambda$ give the
nonzero limiting eigenvalues in
the lower half-plane after
multiplication by $-i$.
The negative eigenvalues give limits
in the upper half-plane, while the
remaining scaled eigenvalues tend
to zero. Consequently,
\[
-\lim_{\varepsilon\downarrow0}
\varepsilon\Im m_{S,T}
(\lambda+i\varepsilon)
=\Tr(P_\lambda SP_\lambda)_+.
\]
On the other hand, \eqref{eq:6} gives
\[
\begin{aligned}
&-\varepsilon\Im m_{S,T}
(\lambda+i\varepsilon) =
\int_{\mathbb R}
\frac{\varepsilon^2}
{(\lambda-v)^2+\varepsilon^2}
\,d\nu_{S,T}(v)
\longrightarrow
\nu_{S,T}(\{\lambda\}).
\end{aligned}
\]
This proves \eqref{eq:9}.
Since $S(zI-T)^{-1}$ is locally
bounded near every real
$v\notin\spec T$, there are no
other atoms.

For $v\notin\spec T$, the
characteristic polynomial of
$S(vI-T)^{-1}$ has real
coefficients, since it is
proportional to
\[
\det(S-\kappa(vI-T)).
\]
Thus, its nonreal eigenvalues occur
in conjugate pairs. Taking boundary
values from the upper half-plane,
we obtain, for almost every such $v$,
\[
-\Im m_{S,T}(v+i0)
=\frac12\sum_{j=1}^d
\left|
\Im\lambda_j\bigl(S(vI-T)^{-1}\bigr)
\right|.
\]
The Stieltjes inversion formula
therefore gives \eqref{eq:10}.

It remains to show that
$\nu_{S,T}$ has no singular
continuous part. The function
$m_{S,T}$ is algebraic. Indeed,
if $\lambda_1(z),\ldots,\lambda_d(z)$
are the eigenvalues of
$S(zI-T)^{-1}$, then $m_{S,T}(z)$
satisfies
\[
\prod_{\substack{
J\subseteq\{1,\ldots,d\}\\
|J|=r}}
\left(
u-\sum_{j\in J}\lambda_j(z)
\right)=0.
\]
The coefficients of this polynomial
in $u$ are symmetric functions of
the eigenvalues and hence rational
functions of $z$. When $r=0$,
the polynomial is simply $u$.

After removing repeated factors,
the algebraic branches have only
finitely many possible poles and
branch points, together with the
zeros and poles of the leading
coefficients. Outside this finite
set, the boundary values of
$m_{S,T}$ are real-analytic.
Stieltjes inversion then shows
that $\nu_{S,T}$ is absolutely
continuous on each of the
corresponding open intervals.
Any remaining singular part is
therefore supported on a finite
set. Such a measure is purely
atomic, and its atoms have already
been determined by \eqref{eq:9}.
Thus, $\nu_{S,T}$ has no singular
continuous part.
\end{proof}

\begin{remark}\label{rem:2.5}
The Herglotz representation used above is classical;
see \cite{GesztesyTsekanovskii}. Related realization
results for matrix-valued Herglotz--Nevanlinna functions
were obtained by Belyi et al.\
\cite{BelyiHassiDeSnooTsekanovskii}.
We thank Professor Sergey Belyi for bringing this work
to our attention and for helpful discussions of
Lemma~\ref{lem:2.4}.

The density formula \eqref{eq:10} is closely related
to the quotient-pencil formulas for tracial joint
spectral measures in
\cite[Theorems~1.7 and 3.1]{HeinTracial}.
A positive planar measure also appears in
\cite{HeinConv}, where the problem concerns a
fourth-order comparison between classical and free
convolutions. In the present paper,
Lemma~\ref{lem:2.4} identifies both the absolutely
continuous part and the atomic masses of the
Herglotz measure. Removing the atoms corresponding
to spectral pinching gives a positive measure
representing the second-order trace loss.
This is the main step in the proof of
Theorem~\ref{thm:2.1}.
\end{remark}

\begin{proof}[Proof of Theorem~\ref{thm:2.1}]
Let $A,B$ be Hermitian and put $D=E_A(B)$.

\emph{The marginal identity.} Fix $b\in\mathbb R$ and apply
Lemma~\ref{lem:2.4} with $S=B-bI$ and $T=A$.
Since
\[
(B-bI)(A-aI)^{-1}
=-(B-bI)(aI-A)^{-1},
\]
formula \eqref{eq:10} gives
\[
k_{A,B}(a,b)
=\frac{d\nu_{B-bI,A}^{\mathrm{ac}}}{da}
\]
for almost every $a\in\mathbb R$. By \eqref{eq:9}, the total mass of its atomic
part is
\[
\begin{aligned}
&\sum_{\alpha\in\spec A}
\Tr\bigl(P_\alpha(B-bI)P_\alpha\bigr)_+ =\Tr(D-bI)_+.
\end{aligned}
\]
Since $\nu_{B-bI,A}$ has total mass
$\Tr(B-bI)_+$ and no singular continuous
part, we obtain
\begin{equation}\label{eq:12}
\int_{\mathbb R}k_{A,B}(a,b)\,da
=\Tr(B-bI)_+-\Tr(D-bI)_+.
\end{equation}

We next determine the support of $\zeta_{A,B}$.
If $a$ lies outside the spectral interval
of $A$, then $A-aI$ is definite.
Similarly, if $b$ lies outside the spectral
interval of $B$, then $B-bI$ is definite.
In either case, the matrix $(B-bI)(A-aI)^{-1}$ 
is similar to a Hermitian matrix.
Thus, all its eigenvalues are real, and
$k_{A,B}(a,b)=0$.

It follows that $\zeta_{A,B}$ is supported in
\[
[\lambda_{\min}(A),\lambda_{\max}(A)]
\times
[\lambda_{\min}(B),\lambda_{\max}(B)].
\]
Moreover, $k_{A,B}$ is measurable, since
the sum in \eqref{eq:1} depends continuously
on $(a,b)$ away from the finitely many lines
$a\in\spec A$.

To compute the total mass, let
$L=\lambda_{\min}(B)$ and
$U=\lambda_{\max}(B)$.
Since $E_A$ is unital and positive,
the spectrum of $D$ also lies in $[L,U]$.
Integrating \eqref{eq:12} with respect
to $b$ and using Tonelli's theorem gives
\[
\begin{aligned}
\zeta_{A,B}(\mathbb R^2)
&=\int_L^U
\bigl[\Tr(B-bI)_+-\Tr(D-bI)_+\bigr]\,db\\
&=\frac12\bigl[
\Tr(B-LI)^2-\Tr(D-LI)^2
\bigr]\\
&=\frac12\bigl(\Tr B^2-\Tr D^2\bigr)\\
&=\frac12\Tr(B-D)^2\\
&=\frac12\|B-D\|_2^2.
\end{aligned}
\]
Here we used $\Tr B=\Tr D$ and the fact
that $E_A$ is the orthogonal projection
onto the commutant of $A$ with respect
to the Hilbert--Schmidt inner product.

Thus, $\zeta_{A,B}$ is finite and positive,
and \eqref{eq:3} holds. In particular,
\[
\zeta_{A,B}=0
\quad\Longleftrightarrow\quad
B=E_A(B)
\quad\Longleftrightarrow\quad
AB=BA.
\]

\emph{The representation formula.}
We first consider the case $x=0$ and $y=1$.
For every $v\in[L,U]$, Taylor's formula gives
\[
f(v)=f(L)+f'(L)(v-L)
+\int_L^U(v-b)_+f''(b)\,db.
\]
Applying functional calculus to $B$ and $D$,
taking traces, and using $\Tr B=\Tr D$, from \eqref{eq:12}
we obtain 
\begin{equation}\label{eq:13}
\Tr f(B)-\Tr f(D)
=\int_{\mathbb R^2}f''(b)\,
d\zeta_{A,B}(a,b).
\end{equation}

We now extend this identity to all
$x,y\in\mathbb R$.
For every $t\in\mathbb R$, we have
\[
\begin{aligned}
&(B+tA-(b+ta)I)(A-aI)^{-1} =(B-bI)(A-aI)^{-1}+tI.
\end{aligned}
\]
Adding a real scalar multiple of the
identity does not change the imaginary
parts of the eigenvalues. Consequently,
\begin{equation}\label{eq:14}
\begin{aligned}
k_{A,B+tA}(a,b+ta)&=k_{A,B}(a,b),\\
E_A(B+tA)&=D+tA.
\end{aligned}
\end{equation}
Applying \eqref{eq:13} to $B+tA$
and using \eqref{eq:14}, we obtain
\[
\Tr f(B+tA)-\Tr f(D+tA)
=\int_{\mathbb R^2}f''(b+ta)\,
d\zeta_{A,B}(a,b).
\]

Suppose $y\ne0$. Set $t=x/y$ and
replace $f(v)$ by $g(v)=f(yv)$.
Since $g''(v)=y^2f''(yv),$ 
we obtain
\[
\begin{aligned}
&\Tr f(xA+yB)-\Tr f(xA+yD) =y^2\int_{\mathbb R^2}
f''(xa+yb)\,d\zeta_{A,B}(a,b),
\end{aligned}
\]
which is \eqref{eq:2}.
The case $y=0$ is immediate.

It remains to justify the stated
domain of $f$. By \eqref{eq:14},
the transformation
\[
(a,b)\longmapsto(a,b+ta)
\]
maps $\zeta_{A,B}$ to
$\zeta_{A,B+tA}$.
Using the support bound proved above,
we obtain, for all $x,y\in\mathbb R$,
\[
\lambda_{\min}(xA+yB)
\le xa+yb
\le\lambda_{\max}(xA+yB)
\]
on $\operatorname{supp}\zeta_{A,B}$.
For $y=0$, the same inequalities follow
directly from the first-coordinate
support bound.

Furthermore,
\[
xA+yD=E_A(xA+yB).
\]
Since $E_A$ is unital and positive,
the spectrum of $xA+yD$ is contained in
\[
[\lambda_{\min}(xA+yB),
 \lambda_{\max}(xA+yB)].
\]
Thus, \eqref{eq:2} holds for every
$C^2$ function defined on an interval
containing the spectrum of $xA+yB$.

In fact, the preceding inequalities
also show that
\[
\operatorname{supp}\zeta_{A,B}
\subseteq W(A,B),
\]
where
\[
W(A,B)=
\left\{
(\langle A\xi,\xi\rangle,
 \langle B\xi,\xi\rangle):
\xi\in\mathbb C^d,\ \|\xi\|=1
\right\}
\]
is the joint numerical range of $A$ and $B$.
Indeed, by the Toeplitz--Hausdorff theorem applied
to $A+iB$, the set $W(A,B)$ is convex.
Its support function is
\[
h_{W(A,B)}(x,y)=\lambda_{\max}(xA+yB),
\]
so the inequalities above imply
$\operatorname{supp}\zeta_{A,B}\subseteq W(A,B)$.

This completes the proof.
\end{proof}

We next show that the measure $\zeta_{A,B}$ is uniquely
determined by the representation formula \eqref{eq:2}.
We also establish its behavior under unitary conjugation
and affine transformations.

\begin{proposition}\label{prop:2.6}
The measure $\zeta_{A,B}$ is the unique compactly
supported finite signed measure on $\mathbb R^2$
satisfying \eqref{eq:2} for all polynomials $f$
and all $x,y\in\mathbb R$.
\end{proposition}

\begin{proof}
Suppose that $\mu$ is another compactly supported
finite signed measure on $\mathbb R^2$ satisfying
\eqref{eq:2} for all polynomials $f$ and all
$x,y\in\mathbb R$.
For $N\ge0$, take $f(t)=t^{N+2}$ in \eqref{eq:2}.
Since $f''(t)=(N+1)(N+2)t^N$, we obtain
\[
\int_{\mathbb R^2}(xa+yb)^N\,d\mu(a,b)
=
\int_{\mathbb R^2}(xa+yb)^N\,d\zeta_{A,B}(a,b)
\]
for every $x\in\mathbb R$ and $y\ne0$.
Both sides are homogeneous polynomials in $x,y$.
Comparing their coefficients gives
\[
\int_{\mathbb R^2}a^pb^q\,d\mu(a,b)
=
\int_{\mathbb R^2}a^pb^q\,d\zeta_{A,B}(a,b)
\]
for all nonnegative integers $p,q$.

Thus, $\mu$ and $\zeta_{A,B}$ have the same
mixed moments. Since both measures have compact
support, there is a compact set $K\subset\mathbb R^2$
containing their supports. By the Stone--Weierstrass
theorem, polynomials are dense in $C(K)$.
Hence
\[
\int_K g\,d\mu
=
\int_K g\,d\zeta_{A,B}
\]
for every $g\in C(K)$.
It follows that $\mu=\zeta_{A,B}$.
\end{proof}

The representation formula and uniqueness also
give the following covariance properties.
For every unitary matrix $U$,
\[
\zeta_{UAU^*,\,UBU^*}=\zeta_{A,B}.
\]
For $\alpha\ne0$ and
$\beta,s,t\in\mathbb R$, we have
\[
\zeta_{\alpha A+sI,\,\beta B+tI}
=\beta^2(T_{\alpha,\beta,s,t})_\#
\zeta_{A,B},
\]
where
\[
T_{\alpha,\beta,s,t}(a,b)
=(\alpha a+s,\beta b+t).
\]
Here $T_\#$ denotes the pushforward of a
measure under the map $T$.
Similarly, the identity \eqref{eq:14} gives
\[
\zeta_{A,B+rA}
=(T_r)_\#\zeta_{A,B},
\qquad
T_r(a,b)=(a,b+ra).
\]
These identities can also be verified directly
from the density \eqref{eq:1}.

We now relate Theorem~\ref{thm:2.1} to
Koplienko spectral shift theory
\cite{GPS,Peller,DykemaSkripka}.

For $x,y\in\mathbb R$, set $H_{x,y}=xA+yD, 
V_y=y(B-D),$ 
and let $L_{x,y}:\mathbb R^2\to\mathbb R$ be
the linear map
\[
L_{x,y}(a,b)=xa+yb.
\]
Define the positive measure
\begin{equation}\label{eq:16}
\kappa_{x,y}
=y^2(L_{x,y})_\#\zeta_{A,B},
\end{equation}
where $(L_{x,y})_\#\zeta_{A,B}$ denotes the
pushforward of $\zeta_{A,B}$ under $L_{x,y}$.

\begin{corollary}\label{cor:2.7}
For every $x,y\in\mathbb R$ and every $C^2$
function $f$ defined on an interval containing
$\spec(xA+yB)$, we have
\begin{equation}\label{eq:15}
\begin{aligned}
&\Tr\bigl[
f(H_{x,y}+V_y)-f(H_{x,y})
-f'(H_{x,y})V_y
\bigr] =\int_{\mathbb R}f''(s)\,
d\kappa_{x,y}(s).
\end{aligned}
\end{equation}
Moreover, $\kappa_{x,y}(\mathbb R)
=\frac12\|V_y\|_2^2.$ 
\end{corollary}

\begin{proof}
Since $H_{x,y}$ commutes with $A$, so does
$f'(H_{x,y})$. Moreover, $E_A(V_y)=0$.
The trace-preserving and bimodule properties
of $E_A$ therefore give
\[
\begin{aligned}
\Tr\bigl(f'(H_{x,y})V_y\bigr)
&=\Tr\bigl(E_A(f'(H_{x,y})V_y)\bigr)\\
&=\Tr\bigl(f'(H_{x,y})E_A(V_y)\bigr)\\
&=0.
\end{aligned}
\]
Since $H_{x,y}+V_y=xA+yB,$ 
the representation formula \eqref{eq:2}
gives \eqref{eq:15} and \eqref{eq:16}.

Finally, by \eqref{eq:3},
\[
\begin{aligned}
\kappa_{x,y}(\mathbb R)
&=y^2\zeta_{A,B}(\mathbb R^2) =\frac{y^2}{2}\|B-D\|_2^2 =\frac12\|V_y\|_2^2.
\end{aligned}
\]
\end{proof}

\section{The corrected BMV conjecture and positive moments}
\label{sec:applications}

For integers $n,m\ge0$, let $S_{n,m}(A,B)$ be the average of all
words containing $n$ copies of $A$ and $m$ copies of $B$, and set
$\mathcal A_{n,m}(A,B)=\Tr S_{n,m}(A,B)$.
For a real polynomial $p(s,t)=\sum_{n,m}c_{nm}s^nt^m$, put
\[
\widetilde p(A,B)=\sum_{n,m}c_{nm}S_{n,m}(A,B).
\]
As before, $D=E_A(B)$.

\begin{proposition}\label{prop:5.1}
For all integers $p,q\ge0$,
\begin{equation}\label{eq:31}
\int_{\mathbb R^2}a^pb^q\,d\zeta_{A,B}(a,b)
=\frac{\mathcal A_{p,q+2}(A,B)-\mathcal A_{p,q+2}(A,D)}
{(q+1)(q+2)}.
\end{equation}
Consequently, for every real bivariate polynomial $p$,
\begin{equation}\label{eq:32}
\Tr\widetilde p(A,B)-\Tr\widetilde p(A,D)
=\int_{\mathbb R^2}\partial_b^2p(a,b)\,d\zeta_{A,B}(a,b).
\end{equation}
\end{proposition}

\begin{proof}
Apply Theorem~\ref{thm:2.1} to $f(v)=v^{p+q+2}$ and compare
the coefficients of $x^py^{q+2}$. Since
\[
\frac{\binom{p+q+2}{q+2}}
{(p+q+2)(p+q+1)\binom{p+q}{q}}
=\frac{1}{(q+1)(q+2)},
\]
we obtain \eqref{eq:31}. Equation~\eqref{eq:32} follows by linearity.
\end{proof}

We now prove the corrected BMV conjecture proposed in
\cite[Conjecture~5.1]{Dinh}, which was previously
established for $m=2$. For $A,B\ge0$,
Theorem~\ref{thm:2.1} gives
$\operatorname{supp}\zeta_{A,B}\subset[0,\infty)^2$.
Together with \eqref{eq:31}, this yields the
conjecture for all word averages.

\begin{theorem}\label{thm:bmv-pinching}
Let $A,B$ be positive semidefinite matrices.
Then, for all integers $n,m\ge0$,
\[
\mathcal A_{n,m}(A,B)\ge
\mathcal A_{n,m}\bigl(A,E_A(B)\bigr).
\]
Equality holds for $m=0,1$. For $m\ge2$,
equality holds if and only if $AB=BA$.
\end{theorem}

\begin{proof}
For $m\ge2$, take $p=n$ and $q=m-2$ in \eqref{eq:31}. This gives
\begin{equation}\label{eq:33}
\mathcal A_{n,m}(A,B)-\mathcal A_{n,m}(A,D)
=m(m-1)\int_{\mathbb R^2}a^nb^{m-2}\,d\zeta_{A,B}(a,b)\ge0.
\end{equation}
For $m=0,1$, equality follows from $E_A(A^n)=A^n$
and the trace-preserving property of $E_A$.
If $AB\ne BA$, then $\zeta_{A,B}(\mathbb R^2)>0$
by \eqref{eq:3}. Since $\zeta_{A,B}$ is absolutely continuous
with respect to planar Lebesgue measure, it gives zero mass
to the coordinate axes. Thus, the integral in \eqref{eq:33}
is strictly positive.
\end{proof}

The polynomial formula \eqref{eq:32} also gives a characterization
of the symmetrized trace inequalities that hold for all
Hermitian matrices.

\begin{theorem}\label{thm:symmetrized-criterion}
For a real bivariate polynomial $p$, the following are equivalent:
\begin{enumerate}
\item[(i)] For all dimensions and all Hermitian matrices $A,B$,
\[
\Tr\widetilde p(A,B)\ge
\Tr\widetilde p\bigl(A,E_A(B)\bigr).
\]
\item[(ii)] $\partial_t^2p(s,t)\ge0$ for all $(s,t)\in\mathbb R^2$.
\end{enumerate}
\end{theorem}

\begin{proof}
The implication (ii)$\Rightarrow$(i) follows from \eqref{eq:32}
and the positivity of $\zeta_{A,B}$.

Conversely, fix $(a_0,b_0)\in\mathbb R^2$ and choose $h,t>0$.
Set
\[
A_h=\begin{pmatrix}a_0-h&0\\0&a_0+h\end{pmatrix},\qquad
B_t=\begin{pmatrix}b_0&t\\t&b_0\end{pmatrix}.
\]
Then $E_{A_h}(B_t)=b_0I$, and \eqref{eq:3} gives
$\zeta_{A_h,B_t}(\mathbb R^2)=t^2$.
By (i) and \eqref{eq:32},
\[
0\le\frac{1}{t^2}\int_{\mathbb R^2}
\partial_b^2p(a,b)\,d\zeta_{A_h,B_t}(a,b).
\]
The normalized measure is a probability measure supported in
$[a_0-h,a_0+h]\times[b_0-t,b_0+t]$.
Letting $h,t\downarrow0$ gives
$\partial_b^2p(a_0,b_0)\ge0$.
Since $(a_0,b_0)$ was arbitrary, (ii) follows.
\end{proof}

We next consider inequalities among the corrected
BMV coefficient differences. For $A,B\ge0$,
set $D=E_A(B)$ and define
\[
M_{p,q}:=
\frac{\mathcal A_{p,q+2}(A,B)
-\mathcal A_{p,q+2}(A,D)}
{(q+1)(q+2)},\qquad p,q\ge0.
\]

\begin{corollary}\label{cor:5.2}
The sequence $(M_{p,q})_{p,q\ge0}$ is a
compactly supported bivariate Stieltjes
moment sequence, represented by
\[
M_{p,q}=
\int_{\mathbb R_+^2}a^pb^q\,
d\zeta_{A,B}(a,b).
\]
For every finite family of pairs $(p_i,q_i)$,
the matrices
\[
[M_{p_i+p_j,q_i+q_j}]_{i,j},\qquad
[M_{p_i+p_j+1,q_i+q_j}]_{i,j},\qquad
[M_{p_i+p_j,q_i+q_j+1}]_{i,j}
\]
are positive semidefinite. In particular,
\[
M_{p+1,q}^2\le M_{p,q}M_{p+2,q},
\qquad
M_{p,q+1}^2\le M_{p,q}M_{p,q+2}.
\]
\end{corollary}

\begin{proof}
The moment representation follows from \eqref{eq:31}.
For any finite family $(c_i)$, set
\[
P(a,b)=\sum_i c_i a^{p_i}b^{q_i}.
\]
Then, for $(r,s)=(0,0),(1,0),(0,1)$,
\[
\begin{aligned}
&\sum_{i,j}\overline{c_i}c_j
M_{p_i+p_j+r,q_i+q_j+s} =\int_{\mathbb R_+^2}
a^rb^s|P(a,b)|^2\,d\zeta_{A,B}(a,b)
\ge0.
\end{aligned}
\]
This proves that all three matrices are positive
semidefinite.

Finally, applying Cauchy--Schwarz to $1,a$ and
$1,b$, respectively, in
$L^2(a^pb^q\,d\zeta_{A,B})$ gives
\[
M_{p+1,q}^2\le M_{p,q}M_{p+2,q},
\qquad
M_{p,q+1}^2\le M_{p,q}M_{p,q+2}.
\]
\end{proof}

For $A,B\ge0$, put $D=E_A(B)$. For $x,y>0$, define
\[
R_{A,B}(x,y)
=\frac{\Tr e^{-xA-yB}-\Tr e^{-xA-yD}}{y^2}.
\]
For each integer $N\ge1$, set
\[
T_N(x,y)=
\Tr\left[\left(
e^{-xA/(2N)}e^{-yB/N}e^{-xA/(2N)}
\right)^N\right]
\]
and
\[
R_{A,B}^{(N)}(x,y)
=\frac{T_N(x,y)-\Tr e^{-xA-yD}}{y^2}.
\]

\begin{corollary}\label{cor:5.3}
For every $x,y>0$,
\begin{equation}\label{eq:34}
R_{A,B}(x,y)
=\int_{\mathbb R_+^2}e^{-xa-yb}\,
d\zeta_{A,B}(a,b).
\end{equation}
Consequently, for all nonnegative integers $r,s$,
\[
(-1)^{r+s}\partial_x^r\partial_y^s
R_{A,B}(x,y)\ge0.
\]
All these inequalities are strict if $AB\ne BA$.

Moreover, the Lie--Trotter approximations satisfy
\[
R_{A,B}^{(N)}(x,y)
\downarrow R_{A,B}(x,y)
\qquad (N\to\infty),
\]
and, for every $N\ge1$,
\[
R_{A,B}^{(N)}(x,y)
\ge R_{A,B}(x,y)
\ge\frac12 e^{-x\|A\|-y\|B\|}
\|B-D\|_2^2.
\]
\end{corollary}

\begin{proof}
Applying Theorem~\ref{thm:2.1} with
$f(v)=e^{-v}$ gives \eqref{eq:34}.
Since $\zeta_{A,B}$ is positive and compactly
supported in $\mathbb R_+^2$, differentiation
under the integral yields
\[
\begin{aligned}
&(-1)^{r+s}\partial_x^r\partial_y^s
R_{A,B}(x,y) =
\int_{\mathbb R_+^2}
a^rb^se^{-xa-yb}\,
d\zeta_{A,B}(a,b)\ge0.
\end{aligned}
\]
If $AB\ne BA$, the measure has positive mass
by \eqref{eq:3} and gives zero mass to the
coordinate axes. Hence the integral is
strictly positive.

To see the monotonicity, let $1\le N<M$ and put
$P=e^{-xA/M}$, $Q=e^{-yB/M}$ and $r=M/N\ge1$.
The Araki--Lieb--Thirring inequality \cite{Araki1990} gives
\[
\Tr\bigl[(P^{r/2}Q^rP^{r/2})^N\bigr]
\ge\Tr\bigl[(P^{1/2}QP^{1/2})^M\bigr],
\]
that is, $T_N(x,y)\ge T_M(x,y)$.
The Lie--Trotter product formula gives
\[
\lim_{N\to\infty}T_N(x,y)
=\Tr e^{-xA-yB}.
\]
It follows that
\[
R_{A,B}^{(N)}(x,y)
\downarrow R_{A,B}(x,y).
\]

Finally, since
\[
\operatorname{supp}\zeta_{A,B}
\subseteq[0,\|A\|]\times[0,\|B\|],
\]
we obtain from \eqref{eq:34}
\[
R_{A,B}(x,y)
\ge e^{-x\|A\|-y\|B\|}
\zeta_{A,B}(\mathbb R^2).
\]
The mass identity \eqref{eq:3} completes
the proof.
\end{proof}

\section{Relation with tracial joint spectral measures}
\label{sec:joint-spectral}

To finish this paper, we compare $\zeta_{A,B}$ with the tracial joint spectral measure
of Hein\"avaara~\cite{HeinTracial}. For Hermitian $A,B$, that theory
provides a positive measure $\mu_{A,B}$ satisfying
\[
\Tr H(f)(xA+yB)
=\int_{\mathbb R^2}f(xa+yb)\,d\mu_{A,B}(a,b),
\qquad
H(f)(v)=\int_0^1\frac{1-t}{t}f(tv)\,dt,
\]
for the admissible test functions. Taking $f(v)=v^k$,
$k\ge1$, and comparing coefficients gives
\begin{equation}\label{eq:35}
\int_{\mathbb R^2}a^nb^m\,d\mu_{A,B}(a,b)
=\frac{\mathcal A_{n,m}(A,B)}{(n+m)(n+m+1)},
\qquad n+m\ge1.
\end{equation}

For the rest of this section, assume $A,B\ge0$ and put
$D=E_A(B)$. The measure $\mu_{A,B}$ is supported on
$K=[0,\|A\|]\times[0,\|B\|]$; see
\cite[Lemma~1.6]{HeinTracial} for the coordinate support
and the defining identity above for the upper bounds.
To obtain a finite measure, define the positive functional
\begin{equation}\label{eq:36}
J_{A,B}(F)=\int_{\mathbb R^2}
\left[\int_0^b(b-s)F(a,s)\,ds\right]d\mu_{A,B}(a,b),
\qquad F\in C(K).
\end{equation}
By \eqref{eq:35},
\[
|J_{A,B}(F)|\le\frac{\|F\|_\infty}{2}
\int_{\mathbb R^2}b^2\,d\mu_{A,B}
=\frac{\|F\|_\infty}{12}\Tr B^2.
\]
Thus, $J_{A,B}$ is represented by a finite positive measure,
which we also denote by $J_{A,B}$. Define the finite signed
comparison measure
\[
\eta_{A,B}=J_{A,B}-J_{A,D}.
\]
For $0<t<1$, set $S_t(a,b)=(ta,tb)$.

\begin{theorem}\label{thm:5.9}
The comparison measure $\eta_{A,B}$ is positive and satisfies
\begin{equation}\label{eq:37}
\eta_{A,B}=
\int_0^1t(1-t)(S_t)_\#\zeta_{A,B}\,dt.
\end{equation}
It is absolutely continuous, with density
\begin{equation}\label{eq:38}
q_{A,B}(a,s)=\int_0^1\frac{1-t}{t}
\,k_{A,B}(a/t,s/t)\,dt
\end{equation}
for almost every $(a,s)\in\mathbb R^2$.
Furthermore,
\begin{equation}\label{eq:39}
\eta_{A,B}(\mathbb R^2)
=\frac1{12}\|B-D\|_2^2.
\end{equation}
In particular, $\eta_{A,B}=0$ if and only if $AB=BA$.
\end{theorem}

\begin{proof}
Write
\[
G_{p,m}=\mathcal A_{p,m}(A,B)-\mathcal A_{p,m}(A,D).
\]
By \eqref{eq:31},
\[
\int_{\mathbb R^2}a^pb^q\,d\zeta_{A,B}(a,b)
=\frac{G_{p,q+2}}{(q+1)(q+2)}.
\]
On the other hand, \eqref{eq:35} and \eqref{eq:36} give
\[
\int_{\mathbb R^2}a^ps^q\,d\eta_{A,B}(a,s)
=\frac{G_{p,q+2}}
{(p+q+2)(p+q+3)(q+1)(q+2)}.
\]
The right side of \eqref{eq:37} has the same mixed moments,
since
\[
\int_0^1t^{p+q+1}(1-t)\,dt
=\frac1{(p+q+2)(p+q+3)}.
\]
Both measures are compactly supported and finite, so
polynomial approximation proves \eqref{eq:37}.
The Jacobian of $S_t$ equals $t^2$, and a change of variables
in \eqref{eq:37} gives \eqref{eq:38}.
Finally, \eqref{eq:3} and
$\int_0^1t(1-t)\,dt=1/6$ give \eqref{eq:39}.
The last assertion follows from \eqref{eq:3}.
\end{proof}

\end{document}